\documentclass[11pt]{article}

\usepackage{amsmath,theorem,amssymb,amsfonts,graphicx,color,mathdots}

\newtheorem{theorem}{Theorem}[section]

\newtheorem{corollary}[theorem]{Corollary}

\newenvironment{proof}{\paragraph{Proof:}}{\hfill$\square$}

\newcommand{\RR}{\mathbb{R}}

\newcommand{\dst}{\displaystyle}

\newcommand{\qand}{\quad \mbox{and} \quad}

\newcommand{\pd}[2]{\frac{\partial #1}{\partial #2}}

\begin{document}
\title{Optimal Venttsel Boundary Control of Parabolic Equations}
\author{Yousong Luo \\
{\small School of Mathematical and Geospatial Sciences, } \\ {\small RMIT University,
GPO Box 2476V} \\ {\small Melbourne, Vic. 3001, AUSTRALIA} \\
 email: {\small yousong.luo@rmit.edu.au}
}


\date{}
\maketitle

\begin{abstract}
In this paper we study the optimality condition for the Venttsel boundary control of a parabolic equation, that is, the state of the dynamic system is governed by a parabolic equation together with an initial condition while the control is applied to the system via the Venttsel boundary condition. A first order necessary condition is derived for the optimal solution in the case of both unconstrained and constrained problems.  The condition is also  sufficient for the unconstrained problem.
\end{abstract}

\section{Introduction}

In this paper we discuss the necessary optimality conditions for a class of optimal control problems formulated as follows:
\begin{equation}\label{prob}  \qquad \left\{ \begin{array}{l}  \mbox{Minimize} \; \; J(u) \\
u \in U \\
F_i(y_u) = 0, 1\leq i \leq k;  F_i(y_u) \leq 0, k+1\leq i \leq m.
\end{array}
\right.
\end{equation}
where $u$ is the control chosen from an allowable set $U$, $y_u$ is the output state variable governed by a state equation corresponding to the input $u$, $J(u)$ is the objective function and $F_i(y_u)$ are the constraint functions on the state variable. 

The state equation under our consideration is an initial-boundary value problem of parabolic equations where the control $u$ is applied to the dynamic system via the Venttsel boundary condition.  We will give the detailed introduction of the state equation, objective and constrain functions in Section 2.

The optimal control problems of systems governed by partial differential equations are well studied. The second order necessary and sufficient optimality conditions for elliptic problems are obtained by Casas E. and Tr\"{o}ltzsch F. in \cite{casas1} and \cite{casas2}.  Similar conditions for parabolic  problems are studied by  Raymond J. and Tr\"{o}ltzsch F. in \cite{Raymond}, Krumbiegel K. and Rehberg J. in  \cite{Krum}.  The general theory on PDE control problems can be found in standard textbooks such as \cite{lions} or Raymond on-line lecture notes.  In those literatures, either the interior distributed control or boundary controls through Dirichlet, Neumann and general oblique boundary conditions are well studied.  This paper will contribute the Venttel boundary control to the existing theory in this field.  

An initial-boundary value problem of a parabolic equation with a parabolic Venttsel boundary condition arises in the engineering problem of heat conduction. A simple example is the problem of heat conduction in a medium enclosed by a thin skin and the conductivities of the medium and the surrounding skin are significantly different, see \cite{cj} and \cite{luo:4}. Generally speaking, all physical phenomena involving a diffusion process along the boundary manifold will give rise to a Vettsel
type boundary condition as it gives rise to a second order tangential derivatives (diffusion) as well as the first and zero order derivatives the unknown function. The theoretical frame work in dealing with such a boundary problems has been developed since 1990's. It was started with elliptic equations by Luo Y. and Trudinger N. in \cite{luo:2} and \cite{luo:3} and continued with parabolic equations by Apushkinskaya D. and Nazarov A. in \cite{apush:1} , \cite{apush:2} and \cite{apush:3}.  The existence, uniqueness as well as the \emph{a priori} estimates of both classical and distributional solutions are established.  It has been shown in \cite{ven} that the Venttsel boundary condition is the most general feasible boundary condition for a parabolic or elliptic equation and, in the degenerate case where the second order term vanishes, it includes Dirichlet, Neumann and general oblique boundary conditions as special cases. 

For the optimal control problems involving Venttsel boundary condition, a first order necessary condition is derived by the author in \cite{luo:1} where the state equation is an elliptic equation, based on the results in \cite{ben}, \cite{bon} and \cite{casas1}.  This paper is the continuation of \cite{luo:1} and it is also an analogue to \cite{casas1}, \cite{casas2} and \cite{Raymond} because similar results have already been obtained for elliptic problems or parabolic problems with traditional boundary conditions.

In the following, we will first state the problems clearly and collect all relevant back ground results for the solutions of our state equation in Section 2. In Section 3 we will establish the differentiability of the objective functional and derive a formula to express the derivatives of it. In Section 4, we will give the optimal condition for both unconstrained and constrained problems. Finally in Section 5 we will make some comments on further development.

\section{Preliminary}

\subsection{Notation, function spaces}

Let $\Omega$ be a bounded open subset of $\RR^n$ with a $C^{3}$
boundary $\Gamma = \partial \Omega$.  For $ T > 0$, we define $Q =\Omega \times (0, T)$ and 
$\Sigma  =\Gamma \times (0, T)$.  For functions $y : ­Q \rightarrow \RR$ the notation $D_i y$ denotes the partial derivative with respect to the space variable $x_i$ and  $D_t y$ denotes the partial derivative with respect to the time variable $t$.  $Dy = (D_1 y, \ldots, D_n y)$ is the gradient of $y$. We also denote by 
\[ D^{\sigma} y : = \frac{\partial ^{|\sigma|} y}{\partial x_1^{\sigma_1}
\partial x_2^{\sigma_2} \cdots \partial x_n^{\sigma_n}} 
\]
the partial derivatives of order $|\sigma|$ where $\sigma = ( \sigma_1, 
\sigma_2,\cdots , \sigma_n) \in \mathbb{N}_{0}^{n}$ is a multi-index of modulus $|\sigma| = \sum_{i=1}^{n}\sigma_i$. The parabolic distance between the points $P_1(x_1,t_1)$ and $P_2(x_2,t_2)$ in $Q$ is defined by $d(P_1, P_2)= (|x_1-x_2|^2+|t_1-t_2|)^{1/2}$ where $|x|$ is the Euclidean norm in $\RR^n$.    When a point $P$ is on the boundary $\Sigma$ we usually write its coordinate as $(s,t)$ where $s$ is a variable on $\Gamma$ and if an integral is involved we will use $ds$ to denote the surface area element of $\Gamma$.
 
The space $C (\overline{Q})$ is the Banach space of all continuous functions $y$ in $\overline{Q}$ with the norm 
    \[ \| y \|_Q = \sup_{Q} |y|. \]
For $ 0 < \alpha < 1$, the space  $C^{\alpha} (\overline{Q})$ is the Banach space of functions with the norm
	\[ \| y \|_{C^{\alpha} (\overline{Q})} =  \| y \|_Q +   [ y ]_{\alpha, Q} \] 
where  $[ \cdot ]_{\alpha, Q}$ stands for the H\"{o}lder semi-norm
\[ [ y ]_{\alpha, Q} := \sup_{P_1, P_2 \in Q} \frac{| y(P_1)-y(P_2)|}{d(P_1, P_2)^{\alpha}}.\] 
The space  $C^{2,\alpha} (\overline{Q})$ is the Banach space of functions with the norm
\[ \| y \|_{C^{2,\alpha} (\overline{Q})} = \sum_{|\sigma|\leq 2} \|  D^{\sigma} y \|_Q + \|D_t y \|_Q +  \sum_{|\sigma|= 2} [ D^{\sigma} y ]_{\alpha, Q} +  [ D_t y ]_{\alpha, Q} .\] 

The restriction of $C^{2,\alpha} (\overline{Q})$ functions on the boundary of $\Sigma$ is denoted by $C^{2,\alpha} (\Sigma)$.  When the functions are independent of $t$ we can define the spaces  $ C^{2, \alpha} (\overline{\Omega})$ and $ C^{2, \alpha} (\Gamma)$ in exactly the same way.
Notice also that every  $ C^{2, \alpha} (\Sigma)$ function can always be extended to a $C^{m, \alpha} (\overline{Q})$ function and such an extension can be carried out in a manner that preserves the norm, i.e. the corresponding $ C^{2, \alpha} (\Sigma)$ norm and $C^{2, \alpha} (\overline{Q})$ are equivalent. Based on such an observation we will not distinguish the spaces $ C^{2, \alpha}(\Sigma)$ and $ C^{2, \alpha} (\overline{Q})$.

Let $\nu=(\nu^1,\ldots,\nu^n)$ be the outward unit normal vector field of $\Gamma$.  Then the outward normal derivative of $y$, denoted by $\partial_{\nu} y$, is defined by
 \[ \partial_{\nu}y = D y \cdot \nu \]
where $Dy$ is the gradient vector of $y$.  Now we define the tangential differential operators. Let $\{c^{ik}\}_{n\times n}$ be the matrix whose entries are given by
 \[c^{ik} =\delta^{ik}-\nu^i \nu^k,\]
where $ \delta^{ik} $ is the Kronecker symbol. Then the first and the second order tangential differential operators are then defined by
 \[\partial_i =c^{ik}D_k,\quad \partial_{ij} =\partial_{i} \partial_{j}
 ,\quad i,j,=1,\cdots,n,\]
hence the tangential gradient operator is defined by
 \[\partial =(\partial_1, \cdots,\partial_n).\]
In particular the Laplace-Beltrami operator on the boundary manifold is then defined by
 \[ \Delta _{\Gamma} = \partial_{i} \partial_{i}. \]
All repeated indecis above indicate a summation from $1$ to $n$.  Note that the second order tangential derivatives so defined are not symmetric in general.

\subsection{State equations, objective functionals and constraints}

The state equation in this paper is the following semi-linear initial-boundary value problem of heat equation
\begin{equation} \label{state} 
(\mbox{SE}) \quad \left\{ \begin{array}{ll} 
\dst D_t y-\Delta y = f    & \quad   \mbox{in} \; Q, \\
 \dst    D_t y -  \Delta_{\Gamma} y +  \partial_{\nu} y = \varphi(s, t, y, u), & \quad \mbox{on} \;   \Sigma , \\
     y(x, 0) = y_0  & \quad \mbox{in} \;  \Omega. 
     \end{array} \right. 
    \end{equation}
where $f \in C^{\alpha}(Q)$ is a given function (a given source of temperature),  $y_0 \in C^{2, \alpha}(\Omega)$ is the initial temperature, $u \in C^{ \alpha}(\Sigma)$ is the control function and $\varphi$ is a given smooth function.  We will not precisely specify the class that $\varphi$ belongs to and assume that all derivatives needed exist and are bounded as long as all the variables in $\varphi$ are bounded. However we make the following general assumption on  $\varphi$ throughout this paper:
\begin{equation} \label{phiassum}
\pd{\varphi}{y} \leq 0.
\end{equation} The boundary condition of this kind is known as the Venttsel boundary condition.  As mentioned in the introduction, physically the Venttsel boundary condition occurs when the boundary manifold $\Gamma$ and the domain $\Omega$ have significantly different conductivity.  In such a case the boundary condition should take the form of 
	\[  D_t y - \kappa \Delta_{\Gamma} y +  \partial_{\nu} y = \varphi(s, t, y, u) \] 
where $\kappa$ is a positive constant not equal to $1$ if we normalize the heat equation in $Q$ to the form in (\ref{state}). However this does not cause any difference in the following theoretical development.  Due to such a reason we only consider the state equation in the form of (\ref{state}). 

The objective functional  $J :  C^{\alpha} (\Sigma) \rightarrow \RR$ is given by \begin{equation}\label{objective}  J(u)=\int_{Q} p(x,t, y_u(x, t)) \; dx dt +\int_{\Sigma} q(s, t, y_u(s, t),u(s, t)) \; ds dt 
\end{equation}
where $p : Q \times \RR \rightarrow \RR$ and $q :
\Sigma \times \RR \times \RR \rightarrow \RR$ are of class $C^1$ and
$y_u = G(u)$ is the solution of the state equation (\ref{state}) corresponding to the control $u$. 

Let  $u_a$ and $b$ be a given pair of functions in $ C^{\alpha} (\Sigma)$ such that $u_a \leq u_b$. Then the allowable control set $U$ is given by 
\begin{equation} \label{allowu}
 U= \{ u \in C^{\alpha} (\Sigma) \; | \; u_a \leq u \leq  u_b\}. 
\end{equation}

The constraint functionals on the state $y_u $ is given by
\begin{equation}\label{constrain} F_i(u) =\int_{Q} a_i(x, t, y_u(x, t)) \; dx dt +\int_{\Sigma} b_i(s, t, y_u(s, t)) \; ds dt
\end{equation}
where $a_i : Q \times \RR \rightarrow \RR$ and $b_i :
\Sigma \times \RR  \rightarrow \RR$ are of class $C^1$.

In the following, when the dependence of a function on the space-time variable $(x, t) \in Q$ or $(s,t) \in \Sigma$ is clear we will simply use a dot `` $\cdot$'' to denote the variable.  For example, we will write $\varphi (s, t, y,u) $ as $\varphi (\cdot, y,u) $

\subsection{Solutions to the state equation}

The existence, uniqueness and {\em a priori}  $ C^{2, \alpha}(\overline{Q})$ norm of the solution to the state equation are all needed in the study of optimal conditions.  The Venttsel problems of parabolic equations have been well studied  \cite{apush:1} and \cite{apush:2} where the existence of classical and distributional solutions are obtained under very general structure conditions for a class of quasilinear equations and boundary conditions.  Since our main focus here is the optimal control problem we will not quote the general existence results from \cite{apush:1} and \cite{apush:2} but will reformat the Theorem to cover only our equations.
\begin{theorem} \label{exist}
Let the following conditions hold: (a) The domain $\Omega$ has a $C^3$ boundary $\Gamma$; 
(b) $f \in C^{\alpha}(\overline{Q})$; (c) $\varphi$ is a bounded smooth function.  
Then for each $u \in C^{ \alpha}(\Sigma)$ the problem (\ref{state}) has a solution $y \in C^{2, \alpha}(\overline{Q})$.	
\end{theorem}
The uniqueness of the classical solution is not given in \cite{apush:1} and \cite{apush:2} for the general problem. We will prove it for our problem by using the weak maximum principle and the arguments in \cite{friedman}.
 \begin{theorem} \label{unique} Under the assumptions of Theorem \ref{exist} and (\ref{phiassum}) the solution $y \in C^{2, \alpha}(\overline{Q})$ to the problem (\ref{state}) is unique.	
 \end{theorem}
 
 \begin{proof}
 	Suppose $y_1$ and $y_2$ are two solutions of (\ref{state}).  Then $y_3=y_1 - y_2$	satisfies
 	\[
 	D_t y_3-\Delta y_3 = 0    \quad   \mbox{in} \; Q,  \qquad 
 	D_t y_3 -  \Delta_{\Gamma} y_3 +  \partial_{\nu} y_3 = \sigma y_3, ,  \quad \mbox{on} \;   \Sigma,  \qquad 
 	y_3(x, 0) = 0   \quad \mbox{in} \;  \Omega  
 	\] 
 	where
 	\[ \sigma = \int_{0}^{1} \frac{\partial \varphi }{\partial y}
 	(\cdot, \tau y_1+(1-\tau)y_2, u)\; d\tau \leq 0 . \]
 	By the weak maximum principle, (Theorem 6, Sec. 2 of \cite{friedman}), the positive maximum of $y_3$ must be obtained at a point $P \in \Sigma$. Also,  by Theorem 14, Sec. 2.2 of \cite{friedman}  $\partial_{\nu} y_3(P) >0$ if $P$ has the inside strong sphere property which is satisfied by our assumption $\Gamma \in C^3$.  However at $P$ we also have $ D_t y_3(P) \geq 0$, $\Delta_{\Gamma} y_3(P) \leq 0$ and $ \sigma y_3(P) \leq 0$  which gives a contradiction to the boundary condition.  The same argument applies to the negative minimum of $y_3$.  Therefore $y_3 \equiv 0$.
 \end{proof}
 
To obtain the {\em a priori} bound for the $C^{2, \alpha}(\overline{Q})$ norm of the solution, the following  weak maximum principle for linear problems is the starting point. 
\begin{theorem} \label{weakmax} Assume that the boundary condition is linear, that is,
$\varphi (s, t, y, u) = \sigma (s, t) y + \theta (s, t)$ and there is a constant $M>0$ such that $|\sigma| \leq M$ for all $(s, t) \in \Sigma$.    Then the solution $y$ to the problem (\ref{state}) satisfies 
\begin{equation} \label{maxvalue}
\| y \|_{Q} \leq C (\| y_0 \|_{\Omega} + \| \theta \|_{\Sigma} + \| f \|_{Q} )
\end{equation} 
where $C$ depends only on $T$ and $M$.
\end{theorem}
\begin{proof}
Let 	$v=e^{-(M+1) t} y$. We then have $ \max |y| \leq e^{(M+1)T} \max |v|.$
Assume $v$ attains a positive maximum at $(x_0, t_0)$ in $\overline{Q}$ then we have the following three cases:

Case I: $(x_0, t_0) \in Q $.  Then
\begin{eqnarray}
 D_t v-\Delta v +(M+1)v = e^{-(M+1)t} (D_t y-\Delta y) = e^{- (M+1)t}  f . 
\end{eqnarray}
Since $\Delta v \leq 0$ and $D_t v \geq 0$ at $(x_0, t_0)$, we have
$ v \leq e^{- (M+1)t_0}  f $ and hence $ \max v \leq  \max |f|$.

Case II: $(x_0, t_0) \in \Sigma $.  Then a similar calculation yields
\begin{eqnarray}
 D_t v -  \Delta_{\Gamma} v +  \partial_{\nu} v +(M+1)v - \sigma v = e^{-(M+1)t}\theta . 
\end{eqnarray}
Since $|\sigma| \leq M$, $\Delta_{\Gamma} v \leq 0$, $\partial_{\nu} v \geq 0$ and $D_t v \geq 0$ at $(x_0, t_0)$, we have $ v \leq e^{- (M+1)t_0} \theta $ and hence
$ \max v \leq  \max |\theta |.$

Case III: $(x_0, t_0) \in \Omega $.  Then $t_0=0$, $v = y$ so $ \max v \leq  \max |y_0 |.$
In summary we have
 \[ \max y \leq e^{(M+1)T}  ( \max_{\Omega} |y_0| + \max_{\Omega} |\theta| + \max_{Q} |f|).\]
If $\min y < 0$ then the same argument as above gives 
 \[  - \min y \leq e^{(M+1)T}  ( \max_{\Omega} |y_0| + \max_{\Omega} |\theta|  + \max_{Q} |f|).\]
Thus the Theorem is proved.
\end{proof}

\begin{corollary}\label{firstbound} In addition to the assumptions in Theorem \ref{exist} assume also that there is a constant $M_1 >0$ such that $\dst | \varphi|, |D \varphi| \leq M$ for all $(s, t, y, u)$.  Then the solution $y$ to the problem (\ref{state}) satisfies 
	\begin{equation} \label{maxvalue2}
	\| y \|_{Q} \leq C (M_1 + \| y_0 \|_{\Omega}  + \| f \|_{Q} )
	\end{equation} 
where $C$ depends only on $T$ and $M$.
\end{corollary}
\begin{proof}
As in the proof of Theorem \ref{unique} we can write the boundary condition as 
\[		D_t y -  \Delta_{\Gamma} y +  \partial_{\nu} y = \sigma y + \varphi(\cdot, 0,u) \,   \quad \mbox{where} \; \; \sigma = \int_{0}^{1} \frac{\partial \varphi }{\partial y} (\cdot, \tau y, u)\; d\tau 
\quad \mbox{and} \; \; |\sigma | \leq M. \]
Then (\ref{maxvalue2}) follows from Theorem \ref{weakmax} with $\theta =\varphi(\cdot, 0,u)$.
\end{proof}
	
We also need the {\em a priori} estimate for the $C^{2, \alpha}(\overline{Q})$ norm of the solution.  We formulate the result to cover our simple problem only instead of giving the general result under complicated assumptions.  Then we will briefly outline the idea of the proof instead of giving the detailed proof. 

\begin{theorem} \label{apriori}
Let the assumptions of Theorem \ref{exist} and \ref{unique} hold.  Assume also that 
$|D^2\varphi| \leq M$ for a constant $M_1 >0$. Then there exists a constant $C$ such that the solution $y \in C^{2, \alpha}(\overline{Q})$ to the problem (\ref{state}) satisfies
\begin{equation} \label{estimate}
\| y \| _{ C^{2, \alpha}(\overline{Q})} \leq C (M_1 + \| y_0 \| _{ C^{2, \alpha}(\overline{\Omega})} + \| f \| _{ C^{\alpha}(\overline{Q})} )
\end{equation}	
where $C$ depends on $n$, $\mbox{diam} (\Omega)$, $T$ and $M$.
\end{theorem}

\begin{proof}
Corollary \ref{maxvalue2} provides a bound for $\| y \|_{Q}$. Using this in Theorem 1.3 of \cite{apush:1} we obtain a bound for $ \| D y \| _{ C^{ \alpha}(\overline{Q})}$, in particular, for 
 $ \|  \partial_{\nu} y \| _{ C^{ \alpha}(\overline{Q})}$.  Then we rewrite the boundary condition in the form 
\[		D_t y -  \Delta_{\Gamma} y - \sigma y_3 =\rho\]
where $\rho =  - \partial_{\nu} y + \varphi(\cdot, 0,u) \in C^{ \alpha}(\Sigma)$ and $\sigma$ is as before. This can be regarded as linear parabolic equation on the boundary $\Sigma$.  By the interior estimate, Theorem 5 of Sec. 3.2 of \cite{friedman}, we then have
 \[  \|  y \| _{ C^{2, \alpha}(\Sigma)} \leq C_1 ( \| y_0 \| _{ C^{2, \alpha}(\Sigma)} + \| \rho \| _{ C^{\alpha}(\Sigma)} ).\]
Finally (\ref{estimate}) follows from Theorem 6 of Sec. 3.2 of \cite{friedman}.
\end{proof}

\section{Differentiability}
\label{sect:3}

In order to derive the optimal condition we investigate the differentiability of the functionals involved in the problem and establish the expressions for the derivatives in this section.  For this purpose we start with the {\em principal system} which is an initial-boundary value problem:
 \begin{eqnarray} \label{principalde} D_t y-\Delta y = f    \quad   \mbox{in} \; Q,  \qquad   D_t y-  \Delta_{\Gamma} y +  \partial_{\nu} y = h, \quad \mbox{on} \;   \Sigma, \qquad  y(x, 0) = y_0  \quad \mbox{in} \;  \Omega, 
 \end{eqnarray}
 where $f \in C^{\alpha}(Q)$, $y_0 \in C^{2, \alpha}(\Omega) $ and $h \in C^{\alpha}(\Sigma) $ are given functions.  We call the following system the {\em adjoint problem} of (\ref{principalde}).
  \begin{eqnarray} \label{adjointde} -D_t z-\Delta z = g    \quad   \mbox{in} \; Q,  \qquad -D_t z-  \Delta_{\Gamma} z +  \partial_{\nu} z = r, \quad \mbox{on} \;   \Sigma, \qquad   z(x, T) = z_T  \quad \mbox{in} \;  \Omega,  
  \end{eqnarray} 
where $g \in C^{\alpha}(Q)$, $z_T \in C^{2, \alpha}(\Omega) $ and $r \in C^{\alpha}(\Sigma)$ are given functions. For the pair of system (\ref{principalde}) and (\ref{adjointde}) we have the following relation.

\begin{theorem} \label{adjointeqn}
	Suppose that $y$ is a solution of (\ref{principalde}) and  $z$ is a solution of (\ref{adjointde}).  Then the following formula holds
	\begin{eqnarray}  \label{adjeqn}
	\int_{Q} f z  \; dx dt +\int_{\Sigma} y  r  \; ds dt &=&  \int_{Q} y g  \; dx dt  + \int_{\Sigma} z h     \; ds dt + \int_{\Omega} (y(x,T) z_T -y_0 z(x,0))  \; dx \nonumber \\
	& &   -\int_{\Gamma} (y(x,T) z_T -y_0 z(x,0))  \; ds   
	\end{eqnarray}
	
\end{theorem}

\begin{proof}
	We multiply the differential equation in (\ref{principalde}) by $z$ and integrate both side over $Q$ to get
	\begin{equation} \label{int1} \int_{Q}  z D_t y  \; dx dt -  \int_{Q} z \Delta y   \; dx dt =  \int_{Q} f z  \; dx dt .
	\end{equation} 
	From the integration by parts formula we have 
	\begin{equation} \label{term1}
	\int_{Q}  z D_t y \; dx dt =  \int_{\Omega} \int_{0}^{T}  z D_t y    \; dt dx 
	=  \int_{\Omega} (y(x,T) z_T -y_0 z(x,0))  \; dx  - \int_{Q} y D_t z   \; dx dt. 
	\end{equation}
	For the second term in (\ref{int1}) we have, by using Green's formula and the boundary condition in (\ref{state}),
	\begin{eqnarray} \label{term2}
	\int_{Q}  z \Delta y \; dx dt & = &  \int_{Q} Dy Dz     \; dx dt - \int_{Q} D (z Dy)     \; dx dt \nonumber \\
	& = &  \int_{Q} Dy Dz     \; dx dt - \int_{\Sigma} z \partial_{\nu}y     \; ds dt  \nonumber \\
	& = &  \int_{Q} Dy Dz     \; dx dt + \int_{\Sigma} z( D_t y    -  \Delta_{\Gamma} y ) \; ds dt - \int_{\Sigma} z h     \; ds dt  
	\end{eqnarray}
	Applying integration by parts formula with respect to $t$ and applying the boundary version of Green's identity (see Lemma 16.1 of \cite{GilbTrud:1}) with respect to $s \in \Gamma$  we have
	\begin{eqnarray*} 
		\int_{\Sigma} z( D_t y    -  \Delta_{\Gamma} y ) \; ds dt &=&  \int_{\Gamma} (y(x,T)  z_T -y_0 z(x,0))  \; ds - \int_{\Sigma} y (D_t z   +     \Delta_{\Gamma} z) \; ds dt \\
		&=& \int_{\Gamma} (y(x,T) z_T -y_0 z(x,0))  \; ds +\int_{\Sigma} y  r  \; ds dt - \int_{\Sigma} y  \partial_{\nu} z  \; ds dt .
	\end{eqnarray*}
	The last equation follows from the boundary condition of problem (\ref{adjointde}). Then (\ref{term2}) becomes
	\begin{eqnarray} \label{term21}
	\int_{Q}  z \Delta y \; dx dt & = & \int_{Q} Dy Dz     \; dx dt -  \int_{\Sigma} z h    \; ds dt +\int_{\Sigma} y  r  \; ds dt - \int_{\Sigma} y \partial_{\nu} z \; ds dt  \nonumber \\
	& & + \int_{\Gamma} (y(x,T) z_T -y_0 z(x,0))  \; ds
	\end{eqnarray}
	Substituting  (\ref{term21}) and  (\ref{term1}) into (\ref{int1}) gives
	\begin{eqnarray} \label{int12}
	\int_{Q} f z  \; dx dt &=&  \int_{\Omega} (y(x,T) z_T -y_0 z(x,0))  \; dx  -\int_{\Gamma} (y(x,T) z_T -y_0 z(x,0))  \; ds -\int_{\Sigma} y r  \; ds dt \nonumber \\
	& & - \int_{Q} y D_t z   - \int_{Q} Dy Dz     \; dx dt +  \int_{\Sigma} z h     \; ds dt + \int_{\Sigma} y \partial_{\nu} z \; ds dt  
	\end{eqnarray}
	
	On the other hand,  by multiply the differential equation in (\ref{adjointde}) by $y$ and integrating over $Q$ we have
	\begin{equation} \label{int2} -\int_{Q}  y D_t z  \; dx dt -  \int_{Q} y \Delta z   \; dx dt =  \int_{Q} y g  \; dx dt 
	\end{equation} 
	which is
	\begin{equation} \label{int2b} -\int_{Q}  y D_t z  \; dx dt -  \int_{Q} Dy Dz   \; dx dt + \int_{\Sigma} y \partial_{\nu} z \; ds dt   =  \int_{Q} y g  \; dx dt .
	\end{equation}
	Finally we obtain
	\begin{eqnarray*}
		\int_{Q} f z  \; dx dt +\int_{\Sigma} y  r  \; ds dt &=&  \int_{Q} y g  \; dx dt  + \int_{\Sigma} z h    \; ds dt + \int_{\Omega} (y(x,T) z_T -y_0 z(x,0))  \; dx \nonumber \\
		& &   -\int_{\Gamma} (y(x,T) z_T -y_0 z(x,0))  \; ds     
	\end{eqnarray*}
\end{proof}

In the following, for convenience,  we express the output $y_u$ corresponding to the control $u$ as the image of a mapping
$G: C^{\alpha} (\Sigma) \rightarrow C^{2, \alpha} (Q)$ so that $y_u=G(u)$.

\begin{theorem} \label{diffg} The mapping $y=G(u)$ is twice Fr\'{e}chet
	differentiable.  If $G'(u) \in {\cal L} ( C^{\alpha} (\Sigma),  C^{2, \alpha} (Q))$ and  $G''(u) \in {\cal L} ( C^{\alpha} (\Sigma) \times  C^{\alpha} (\Sigma),  C^{2, \alpha} (Q))$ are the first and second order Fr\'{e}chet derivative of $G$ at $u$, then for each $v, v_1, v_2\in
	C^{\alpha} (\Sigma)$ the function  $z=\langle G'(u), v\rangle$ is the unique solution of the
	boundary value problem
	\begin{eqnarray} \label{gdash1} D_t z-\Delta z  = 0    \quad   \mbox{in} \; Q,  \qquad  \qquad  z(x, 0)  =  0  \quad \mbox{in} \;  \Omega,  \nonumber \\
	D_t z-  \Delta_{\Gamma} z +  \partial_{\nu} z =  \frac{\partial \varphi }{\partial y}
	(\cdot, y, u)  z + \frac{\partial \varphi }{\partial u}
	(\cdot, y, u) v   \quad \mbox{on} \;   \Sigma 
	\end{eqnarray}
	and the function  $z_{12}=\langle G''(u), (v_1, v_2) \rangle$ is the unique solution of the
	boundary value problem
	\begin{eqnarray} \label{gdash2}  D_t z-\Delta z & = & 0    \quad   \mbox{in} \; Q,  \qquad  \qquad  z(x, 0)  =  0  \quad \mbox{in} \;  \Omega,  \nonumber \\
	 D_t z-  \Delta_{\Gamma} z +  \partial_{\nu} z & = & \frac{\partial \varphi }{\partial y}
	(\cdot, y, u)  z + \frac{\partial^2 \varphi }{\partial y^2}
	(\cdot, y, u) z_1 z_2  \nonumber \\ 
	&  & + \frac{\partial ^2 \varphi }{\partial u \partial y}
	(\cdot, y, u) (z_1 v_2 + z_2 v_1)  + \frac{\partial ^2 \varphi }{\partial u ^2}
	(\cdot, y, u) v_1  v_2   \quad \mbox{on} \;   \Sigma 
	\end{eqnarray}
where  $z_i=\langle G'(u), v_i \rangle$ for $i = 1, 2$.
\end{theorem}

\begin{proof}  We first prove that $G$ is Gateaux-differentiable and calculate the G-derivative  $dG(u)$. Let $v\in C^{\alpha} (\Sigma)$ and
consider $y_{\lambda}=G(u+\lambda v)$ and $y=G(u)$.  It follows that 
	\[ z=\langle G'(u), v\rangle = \lim_{\lambda\rightarrow 0} \frac{w_{\lambda}}{\lambda} \]
where $w_{\lambda}=y_{\lambda}-y$ satisfies	
	\begin{eqnarray} \label{eqn1}  D_t w_{\lambda}-\Delta w_{\lambda}  = 0    \quad   \mbox{in} \; Q,  \quad    w_{\lambda}(x, 0)  =  0  \quad \mbox{in} \;  \Omega.  \nonumber \\
	D_t w_{\lambda}-  \Delta_{\Gamma} w_{\lambda} +  \partial_{\nu} w_{\lambda} = \varphi (\cdot, y_{\lambda}, u+\lambda v) - \varphi (\cdot, y, u)  \quad \mbox{on} \;   \Sigma .
	\end{eqnarray}

 Dividing  (\ref{eqn1}) by $\lambda$ we can see that $z_{\lambda}=w_{\lambda}/\lambda$ satisfies
	\begin{eqnarray} \label{eqnzlam} D_t z_{\lambda}-\Delta z_{\lambda}  = 0    \quad   \mbox{in} \; Q,  \quad    z_{\lambda}(x, 0)  =  0  \quad \mbox{in} \;  \Omega.  \nonumber \\
	D_t z_{\lambda}-  \Delta_{\Gamma} z_{\lambda} +  \partial_{\nu} z_{\lambda} = \beta_{\lambda}z_{\lambda}+\gamma_{\lambda} v  \quad \mbox{on} \;   \Sigma .
	\end{eqnarray}
where
	\[ \beta_{\lambda} = \int_{0}^{1} \frac{\partial \varphi }{\partial y}
	(\cdot, \tau y_{\lambda}+(1-\tau)y, u+\lambda v)\; d\tau \qand \gamma_{\lambda} =\int_{0}^{1}
	\frac{\partial \varphi }{\partial u}
	(\cdot, y, u+\tau \lambda v)\; d\tau.\]
We can assume that $\lambda$ is bounded, say $|\lambda| \leq 1$.  Obviously  $\beta_{\lambda} \in C^{2,\alpha} (\Sigma)$.  It follows from Theorem \ref{apriori} that
	\[ \| \beta_{\lambda} \|_{C^{\alpha}(\Sigma)}  \leq  C_1(d+\| y \|_{C^{\alpha}(\Sigma)}+ \| y_{\lambda} \|_{C^{\alpha}(\Sigma)}+\| u \|_{C^{\alpha}(\Sigma)}+\| v \|_{C^{\alpha}(\Sigma)}) \leq C_2
	\]
for a constant $C_2$. Notice also that  $\| \gamma_{\lambda} v \|_{C^{\alpha}(\Sigma)}$ is also bounded and hence Theorem \ref{apriori} implies 
	\[ \| z_{\lambda} \|_{C^{2,\alpha}(\overline{Q})} \leq C_3 (\| z_{\lambda} \|_{Q} +1)\]
for some constants $C_3$. Applying the maximum principle Theorem \ref{maxvalue} we know that $\| z_{\lambda} \|_{Q}$ is bounded.  In summary we 
	\begin{equation}
	 \| z_{\lambda} \|_{C^{2,\alpha}(\overline{Q})} \leq C_4
	\label{ztunif}
	\end{equation}
for a constant $C_4$ independent of $\lambda$. This implies that, up to a subsequence, $z_{\lambda}$ converges to a function $z$ in $C^{2,\alpha} (\overline{Q})$ as $\lambda \rightarrow 0$ and
	\[ \lim_{\lambda \rightarrow 0} \beta_{\lambda} = \frac{\partial \varphi }{\partial y}
	(\cdot, y, u) \qand \lim_{\lambda \rightarrow 0} \gamma_{\lambda} = \frac{\partial \varphi }{\partial u} (\cdot, y, u).\]
By taking limit in (\ref{eqnzlam}) we can see that  $z=\langle dG(u), v\rangle$ is the solution of (\ref{gdash1}).

The uniqueness of $z$ is guaranteed by Theorem \ref{unique}.
	
Next we examine the continuity of $dG$.  Notice that $dG(u)\in
	{\cal L} ( C^{\alpha} (\Sigma),  C^{2, \alpha}
	(\overline{Q})  )$ and
	\[ \| dG(u) \| = \sup_{\| v \| =1} | \langle dG(u), v\rangle |_{C^{2, \alpha}
		(\overline{Q}) }. \]
Therefore to prove the continuity of $dG(u)$ is to prove that as $\tilde{u} \rightarrow u$ in $C^{\alpha} (\Sigma)$ 
	\[ \| dG(\tilde{u})- dG(u)\| = \sup_{\| v \| =1} | \langle dG(\tilde{u}), v\rangle -
	\langle dG(u), v\rangle |_{C^{2, \alpha}
		(\overline{Q})} \rightarrow 0. \]
For any $v \in C^{\alpha} (\Sigma)$ with $\| v \| =\| v \|_{C^{2, \alpha} 	(\overline{Q})}=1$
consider $ \tilde{z}= \langle dG(\tilde{u}), v\rangle$ and $z= \langle dG(u), v\rangle$.  Then $w =\tilde{z}-z$ satisfies
	\begin{eqnarray} \label{contin} 
	D_t w-\Delta w  = 0    \quad   \mbox{in} \; Q,  \qquad  \qquad  w(x, 0)  =  0  \quad \mbox{in} \;  \Omega.  \nonumber \\
	D_t w-  \Delta_{\Gamma} w +  \partial_{\nu} w =  \frac{\partial \varphi }{\partial y}
	(\cdot, \tilde{y}, \tilde{u})  \tilde{z} + \frac{\partial \varphi }{\partial u}
	(\cdot, \tilde{y}, \tilde{u}) v -\frac{\partial \varphi }{\partial y}
	(\cdot, y, u)  z - \frac{\partial \varphi }{\partial u}
	(\cdot, y, u) v   \quad \mbox{on} \;   \Sigma ,
	\end{eqnarray}
where $\tilde{y} = G(\tilde{u})$. 
All we need to show is that $w \rightarrow 0$ in $C^{2,\alpha} (\bar{Q})$ uniformly with respect to $\| v \|=1$, as $\tilde{u} \rightarrow u$ in $C^{\alpha} (\Sigma)$.  To this end we rewrite right hand side of the boundary condition in (\ref{contin}) in the form
$ \zeta w + \eta $ where 
	\[ \zeta = \frac{\partial \varphi }{ \partial y} (\cdot, \tilde{y}, \tilde{u})\]
and 
	\[ \eta =\left( \frac{\partial \varphi }{\partial y}
	(\cdot, \tilde{y}, \tilde{u})- \frac{\partial \varphi }{\partial y}
	(\cdot, y, u)\right)  z + \left(\frac{\partial \varphi }{\partial u}
	(\cdot, \tilde{y}, \tilde{u}) - \frac{\partial \varphi }{\partial u}
	(\cdot, y, u) \right)v.\]
Further more if we put
	\[ \mu_1 =  \int_{0}^{1} \frac{\partial^2 \varphi }{\partial y^2}
	(\cdot, y+\tau (\tilde{y}-y), \tilde{u} )\; d \tau, \quad  \mu_2 =  \int_{0}^{1} \frac{\partial^2 \varphi }{ \partial y \partial u} 	(\cdot, y, u+\tau (\tilde{u}-u))\; d\tau , \]
	\[ \mu_3 =  \int_{0}^{1} \frac{\partial^2 \varphi }{\partial y \partial u}
	(\cdot,  y+\tau (\tilde{y}-y), \tilde{u})\; d \tau \qand \mu_4 =  \int_{0}^{1} \frac{\partial^2 \varphi }{\partial u^2}
	(\cdot, y, u+\tau (\tilde{u}-u))\; d\tau \]
then $\eta$ can be written as
	\[ \eta =(\mu_1 z+ \mu_3 v) (\tilde{y} - y) + ( \mu_2 z+ \mu_4 v )(\tilde{u} - u).\]
From the assumption on $\varphi$ we know that $\mu_1 z+ \mu_3 v$ and $\mu_2 z+\mu_4 v $ are both in   $C^{\alpha} (\overline{Q})$ and hence
	\[ \| \eta \|_{C^{\alpha} (\Sigma)} \leq C_5 ( \| \tilde{u} - u \|_{C^{\alpha} (\Sigma)} +\| \tilde{y} - y \|_{C^{\alpha} (\Sigma)} ).\]
Putting everything together we can write the boundary condition for $w$ as
 \[ D_t w-  \Delta_{\Gamma} w +  \partial_{\nu} w - \zeta w =  \eta  \quad \mbox{on} \;   \Sigma . \]
By Theorem \ref{apriori} we then have
	\[ \| w \|_{ C^{2, \alpha} (\overline{Q})} 
	\leq C_6 ( \| \tilde{u} - u \|_{C^{\alpha} (\Sigma)}  +\| \tilde{y} - y \|_{C^{\alpha} (\Sigma)} ) \rightarrow 0
	\]
which proves the continuity of  $dG(u)$. Finally, since $G(u)$ is continuously Gateaux differentiable,  we conclude that $G(u)$ is also Fr\'{e}chet differentiable and that the Fr\'{e}chet derivative $G'(u)$ is equal to $dG(u)$. 

For the second order derivative we let  $y_{\lambda}=G(u+\lambda v_2 )$ and  $z_{\lambda}=\langle G'(u+\lambda v_2), v_1 \rangle $.  We then have
\[ z_{12}=\langle G''(u), (v_1, v_2) \rangle =  \lim_{\lambda\rightarrow 0} \frac{1}{\lambda}(z_{\lambda}-z_1). \]
By using exactly the same argument above we can prove the existence of  $G''(u)$  and derive  the equation that $ z_{12}$ must satisfy.  Since this is a lengthy but straight forward exercise we omit the details here.
\end{proof}

Now we are in the position to establish the differentiability and express the derivative of the objective functional $ J(u)$. 
\begin{theorem} \label{diffj} The functional $J$ is twice Fr\'{e}chet differentiable and for every $u, v, v_1, v_2 \in   C^{\alpha} (\Sigma)$ and $y = G(u)$ 	we have
\begin{equation} \label{jdash1}
	 \langle J'(u), v\rangle =\int_{\Sigma} \left( \pd{q}{u}(\cdot, y,u)- \frac{\partial \varphi }{\partial u}
	(\cdot, y, u)w \right) v \; ds dt
\end{equation}
and 
\begin{eqnarray} \label{jdash2}
\langle J''(u), (v_1,v_2) \rangle & =& \int_{Q} 	\pd{^2 p}{y ^2}(\cdot, y) z_1 z_2\; dx dt +
 \int_{\Sigma} \left( \pd{^2 q}{ y^2 }(\cdot, y,u)- \frac{\partial^2 \varphi }{\partial y^2}
 (\cdot, y, u)w \right)  z_1 z_2 \; ds dt  \nonumber \\
& & + \int_{\Sigma} \left( \pd{^2 q}{u \partial y }(\cdot, y,u)- \frac{\partial^2 \varphi }{\partial u \partial y} (\cdot, y, u)w \right)  (z_1 v_2+z_2 v_1) \; ds dt   \nonumber \\
& & + \int_{\Sigma} \left( \pd{^2 q}{u ^2}(\cdot, y,u)- \frac{\partial^2 \varphi }{\partial u^2}
(\cdot, y, u)w \right) v_1 v_2 \; ds dt
\end{eqnarray}
where $z_{i} = \langle G'(u), v_i \rangle $ for $i=1, 2$ and  $w$ is the solution of
	\begin{eqnarray} \label{jdash3} -D_t w-\Delta w  = 	\pd{p}{y}(\cdot, y)    \quad   \mbox{in} \; Q,  \qquad    w(x, T)  =  0  \quad \mbox{in} \;  \Omega.  \nonumber \\
		-D_t w-  \Delta_{\Gamma} w +  \partial_{\nu} w =  \frac{\partial \varphi }{\partial y}
		(\cdot, y, u)  w - \pd{q}{y}(\cdot,y, u)  \quad \mbox{on} \;   \Sigma .
	\end{eqnarray}
\end{theorem}

\begin{proof} Define a mapping $H: C^{2,\alpha} (Q) \times  C^{\alpha} (\Sigma) \rightarrow \RR$ by
	\[ H(y,u) =\int_{Q} p(x, t, y(x, t)) \; dx dt
	+\int_{\Sigma} q(s, t, y(s, t),u(s, t)) \; ds dt.
	\]
Obviously $H$ is differentiable with respect to both $y$ and $u$. Also, for every $\tilde{y}$	and $\tilde{u}$ we have
	\[ \langle \pd{H}{y}(y,u), \tilde{y}\rangle = \int_{Q} \pd{p}{y}(x, t, y(x, t)) \tilde{y} \; dx dt
	+\int_{\Sigma} \pd{q}{y}(s, t, y(s, t),u(s,t)) \tilde{y}\; ds dt
	\]
and
	\[ \langle \pd{H}{u}(y,u), \tilde{u}\rangle = \int_{\Sigma} \pd{q}{u}(s, t, y(s, t),u(s, t)) \tilde{u}\;
	ds dt.
	\]
Since $ J(u) = H(G(u),u)$,	by the chain rule we have
\begin{equation} \label{diffh} 
	 \langle J'(u), v \rangle = \langle \pd{H}{y}(y,u)G'(u)+  \pd{H}{u}(y,u), v\rangle 
	= \langle \pd{H}{y}(y,u), G'(u) v\rangle +\langle \pd{H}{u}(y,u), v\rangle 
\end{equation}
where $G'(u) v $ stands for the solution $z= \langle G'(u),  v \rangle$ of (\ref{gdash1}) in Theorem \ref{diffg}. Therefore
	\begin{eqnarray}
	\langle J'(u), v\rangle	& = &  \int_{Q} \pd{p}{y}(\cdot, y)z \; dx dt
	+ \int_{\Sigma} \pd{q}{y}(\cdot, y, u) z \; ds dt
	   + \int_{\Sigma} \pd{q}{u}(\cdot, y, u) v \; ds dt
	\end{eqnarray}

Now we set (\ref{gdash1}) as the principal system and treat (\ref{jdash1}) as its adjoint system.  Let $w$ be the solution of (\ref{jdash1}).  Applying Theorem \ref{adjointeqn} to $z$ and $w$ together with the information $z_0=0$,  $w_T=0 $,
	\[ f=0, \quad g=\pd{p}{y}(\cdot, y), \quad h = \frac{\partial \varphi }{\partial y} (\cdot, y, u)  z + \frac{\partial \varphi }{\partial u}
	(\cdot, y, u) v, \qand r = \frac{\partial \varphi }{\partial y}
(\cdot, y, u)  w - \pd{q}{y}(\cdot,y, u) \]
we obtain
	\begin{eqnarray*} 
		 \int_{\Sigma} z \left( \frac{\partial \varphi }{\partial y}
		 (\cdot, y, u)  w - \pd{q}{y}(\cdot,y, u)\right)      \; ds dt    &=&  \int_{Q} z \pd{p}{y}(\cdot, y)  \; dx dt  \\
		 & & + \int_{\Sigma} w  \left( \frac{\partial \varphi }{\partial y} (\cdot, y, u)  z + \frac{\partial \varphi }{\partial u}
		 (\cdot, y, u) v\right)    \; ds dt   
	\end{eqnarray*}
which is 
	\[ \int_{Q} z \pd{p}{y}(\cdot, y)  \; dx dt = -\int_{\Sigma} z \pd{q}{y}(\cdot,y, u))     \; ds dt   - \int_{\Sigma} w \frac{\partial \varphi }{\partial u}
	(\cdot, y, u) v    \; ds dt  \] 
A substitution of this into (\ref{jdash1}) gives
	\[ \langle J'(u), v\rangle  = \int_{\Sigma} \left( \pd{q}{u}(\cdot, y,u)- \frac{\partial \varphi }{\partial u}
	(\cdot, y, u)w\right) v \; ds dt.
	\]
For the second order derivative we differentiate $ \langle J'(u), v_1 \rangle$ using the formula (\ref{diffh}) to get
	\begin{eqnarray} \label{hdash2}
	\langle J''(u), (v_1, v_2) \rangle	& = & 
	\langle \pd{H}{ y}(y,u), z_{12} \rangle + \langle \pd{^2 H}{y^2}(y,u), z_1 z_2\rangle + \langle \pd{^2 H}{u^2}(y,u), v_1 v_2\rangle \nonumber \\
&&	+\langle \pd{^2H}{u \partial y}(y,u), (z_1 v_2 +z_2 v_1 \rangle  \nonumber \\
&= &	 \int_{Q} \pd{p}{y}(\cdot, y) z_{12} \; dx dt
+ \int_{\Sigma} \pd{q}{y}(\cdot, y, u)  z_{12} \; ds dt + \int_{Q} \pd{^2 p}{y^2}(\cdot, y) z_1 z_2 \; dx dt \nonumber \\
&&	+ \int_{\Sigma} \pd{^2 q}{y^2}(\cdot, y, u)  z_1 z_2 \; ds dt
	+ \int_{\Sigma} \pd{^2 q}{u^2}(\cdot, y, u) v_1 v_2 \; ds dt \nonumber \\
	&& + \int_{\Sigma} \pd{^2 q}{u \partial y }(\cdot, y, u) (z_1 v_1+ z_2 v_2) \; ds dt
	\end{eqnarray}
Then we set (\ref{gdash2}) as the principal system and treat (\ref{jdash3}) as its adjoint system.  Let $z_{12}$ and $w$ be the solutions of (\ref{gdash2}) and (\ref{jdash3}) respectively.  Applying Theorem \ref{adjointeqn} to $z_{12}$ and $w$ together with the information 
\[ z_0=0, \quad  w_T=0 , \quad f=0, \quad g=\pd{p}{y}(\cdot, y), \quad  \qand r = \frac{\partial \varphi }{\partial y} (\cdot, y, u)  w - \pd{q}{y}(\cdot,y, u) 
\]
and 
\[h = \frac{\partial \varphi }{\partial y}
(\cdot, y, u)  z + \frac{\partial^2 \varphi }{\partial y^2}
(\cdot, y, u) z_1 z_2  + \frac{\partial ^2 \varphi }{\partial u \partial y}
(\cdot, y, u) (z_1 v_2 + z_2 v_1)  + \frac{\partial ^2 \varphi }{\partial u ^2}
(\cdot, y, u) v_1  v_2,\]
we obtain
\begin{eqnarray*} 
\int_{Q} z_{12} \pd{p}{y}(\cdot, y)  \; dx dt & = & -\int_{\Sigma} z_{12} \pd{q}{y}(\cdot,y, u))     \; ds dt  - \int_{\Sigma}  w  \pd{^2 \varphi }{y^2}(\cdot, y, u)  z_1 z_2 \; ds dt  \nonumber \\
&&  \int_{\Sigma}  w  \pd{^2 \varphi}{u^2}(\cdot, y, u) v_1 v_2 \; ds dt - \int_{\Sigma}  w  \pd{^2 \varphi}{u \partial y }(\cdot, y, u) (z_1 v_1+ z_2 v_2) \; ds dt
\end{eqnarray*}
which is 
\[ \int_{Q} z \pd{p}{y}(\cdot, y)  \; dx dt = -\int_{\Sigma} z \pd{q}{y}(\cdot,y, u))     \; ds dt   - \int_{\Sigma} w \frac{\partial \varphi }{\partial u}
(\cdot, y, u) v    \; ds dt  \] 
A substitution of this into (\ref{hdash2}) produces the formula  (\ref{jdash2}) for 
$ \langle J''(u), (v_1, v_2) \rangle $. 
\end{proof}

\section{Optimality Condition}

\subsection{Optimization without state constraint}

When the optimization problem doesn't have state constraints but has only the constraint (\ref{allowu}) on the control $u$, a first order necessary optimality condition for $\bar{u}$ to be a solution is that $\langle J'(\bar{u}), v \rangle \geq 0$ for all $v$ chosen from the allowable control set $U$.  In the case $\langle J'(\bar{u}), v \rangle = 0$ for some $v \in U$ then the second order necessary optimality condition is $\langle J''(\bar{u}), (v, v) \rangle \geq 0$ for those $v$. The following necessary optimality condition is a straight forward consequence of Theorem \ref{diffj}.
\begin{theorem} \label{opcondns} The necessary condition for $\bar{u} \in U $ to be an optimal solution of $\inf J(u)$ is 
\begin{equation} \label{opcondns1}
\int_{\Sigma} \left( \pd{q}{u}(s, t, \bar{y},\bar{u})- \frac{\partial \varphi }{\partial u}
(s, t, \bar{y}, \bar{u}) \bar{w} \right) v \; ds dt  \geq 0 
\end{equation}		
for all $v \in U$,  where the couple $(\bar{y}, \bar{w}) $ is the solution of the following system
	\begin{equation} \label{sysopcond}
	\left\{ \begin{array}{l}
 \dst	D_t y-\Delta y = f  \quad  \mbox{in}  \;  Q, \qquad 
 	D_t y-  \Delta_{\Gamma} y +  \partial_{\nu} y =\varphi (\cdot , y, \bar{u})   \quad \mbox{on}  \;   \Sigma , \qquad 
 	y(x, 0) = y_0  \quad  \mbox{in}  \;  \Omega. \\ 
 	 \phantom{x} \\	
 \dst		-D_t w-\Delta w  = 	\pd{p}{y}(\cdot, y)    \quad   \mbox{in} \; Q,  \quad 
 			-D_t w-  \Delta_{\Gamma} w +  \partial_{\nu} w =  \frac{\partial \varphi }{\partial y}
 			(\cdot, y, \bar{u})  w - \pd{q}{y}(\cdot,y, \bar{u})  \quad \mbox{on} \;   \Sigma, \\
 			\quad \quad   w(x, T)  =  0  \quad \mbox{in} \;  \Omega. 
\end{array}	\right.
	\end{equation}	
In the case where the integral in (\ref{opcondns1}) is equal to $0$ for some $v$ then the second order necessary optimality condition is 
 \begin{eqnarray} \label{opcondns2}
 \int_{Q} 	\pd{^2 p}{y ^2}(\cdot, y) z^2\; dx dt +
 \int_{\Sigma} \left( \pd{^2 q}{ y^2 }(\cdot, y,u)- \frac{\partial^2 \varphi }{\partial y^2}
 (\cdot, y, u)w \right)  z^2 \; ds dt  \nonumber \\
 + 2 \int_{\Sigma} \left( \pd{^2 q}{u \partial y }(\cdot, y,u)- \frac{\partial^2 \varphi }{\partial u \partial y} (\cdot, y, u)w \right)  z v \; ds dt   \nonumber \\
 + \int_{\Sigma} \left( \pd{^2 q}{u ^2}(\cdot, y,u)- \frac{\partial^2 \varphi }{\partial u^2}
 (\cdot, y, u)w \right) v^2 \; ds dt \geq 0
 \end{eqnarray}
for such $v$, where $z= \langle G'(\bar{u}), v \rangle$. 
\end{theorem}

As a special case we have:

\begin{corollary} \label{opcondns} If the objective functional $J(u)$ is convex,  then the necessary and sufficient condition for $\bar{u}$ to be an optimal solution is 
	\begin{equation} \label{opconduncon}
	\pd{q}{u}(s, t, \bar{y},\bar{u})- \frac{\partial \varphi }{\partial u}
	(s, t, \bar{y}, \bar{u}) \bar{w} =0 
	\end{equation}		
where the couple $(\bar{y}, \bar{w}) $ is the solution of (\ref{sysopcond}).
\end{corollary}
\begin{proof} 
	In such a case (\ref{opcondns1}) becomes a necessary and sufficient condition with the equality holds true for all $v$.  This implies (\ref{opconduncon}).
\end{proof}

For an application of Corollary \ref{opcondns} we consider the example when $\varphi (\cdot , y, u) =u$ and the objective function is given by
\[  J_1(u)= \frac{1}{2} \int_{Q} (y_u-y_g)^2  \; dx dt+ \frac{\beta}{2} \int_{\Sigma} u^2  \; ds dt
\]
where $\beta > 0$ is a constant and $y_g$ is a given reference temperature. The objective in this example is to minimize the difference between the actual temperature and a given reference temperature plus the cost of the control. To verify that $J(u)$ is also convex with respect to $u$ we let $y_1$, $y_2$ and $y_{\lambda}$ be the solution of the state equation (\ref{state}) corresponding to the controls $u_1$, $u_2$ and $\lambda u_1+(1-\lambda)u_2$ respectively.  It is easy to see that $y_{\lambda}=
\lambda y_1+(1-\lambda)y_2$ from the equation (\ref{state}) and hence the convexity of $J_1(u)$ follows.  Therefore, in such a case, we can find the optimal solution precisely by solving a system of heat equations.
\begin{theorem} If $\varphi (\cdot , y, u) =u$ then the optimal solution $\bar{u}$ of the problem $\inf J_1(u)$ is given by
 \[ \bar{u} =  \frac{1}{\beta} \bar{w} \]
 where the pair $ (\bar{w}, \bar{y})$ is the solution of the following system
	\begin{eqnarray} \label{optcondex}
	- D_t w-\Delta w = y -y_g   \quad  \mbox{in}  \;  Q, \qquad  
	-D_t w-  \Delta_{\Gamma} w +  \partial_{\nu} w = 0  \quad \mbox{on}  \;   \Sigma , \qquad 
	w(x, T) = 0  \quad  \mbox{in}  \;  \Omega. \nonumber \\  
	D_t y-\Delta y = f  \quad  \mbox{in}  \;  Q, \qquad 
	D_t y-  \Delta_{\Gamma} y +  \partial_{\nu} y = \frac{1}{\beta} w   \quad \mbox{on}  \;   \Sigma , \qquad 
	y(x, 0) = y_0  \quad  \mbox{in}  \;  \Omega.  
	\end{eqnarray}
\end{theorem}

\begin{proof} In Theorem \ref{opcondns} we put $p=(y-y_g)^2/2$, $q=(\beta u^2)/2$, $\varphi = u$.  It follows that 
	\[ \pd{p}{y} =y-y_g, \quad \pd{q}{y} = 0, \quad \pd{q}{u} = \beta u,  \quad \pd{\varphi}{u} =1 \qand \pd{\varphi}{y} =0.\]
Then the sufficient and necessary condition in Theorem \ref{opcondns} becomes $ \beta \bar{u} =	\bar{w}$, as long as  $ (\bar{w}, \bar{y})$ is the solution of (\ref{optcondex}).  Therefore $ \bar{u} =  \frac{1}{\beta} \bar{w}$ is the optimal solution.
\end{proof}

\subsection{Constrained optimization}

A function $\bar{u } \in U$  is said to be a local solution, or a
locally optimal control, of (\ref{prob}) if there is a number
$\delta > 0$ such that $J(u) \geq J(\bar{u})$ holds for all $u \in
U$ satisfying $|u-\bar{u}| < \delta$, with their associated state
$y $ and the state constraint on $y$.  Our main result is the first order necessary condition for a $\bar{u } \in U$ to be a local solution.

The differentiability and the expression of the Fr\'{e}chet derivative of the constraints $F_i(u)$ can be obtained as follows.
\begin{equation} F(u) =\int_{Q} a(x, t, y_u(x, t)) \; dx dt +\int_{\Sigma} b(s, t, y_u(s, t)) \; ds dt
\end{equation}
\begin{theorem} \label{difff} The functional $F_i$ is Fr\'{e}chet differentiable and for every $u, v \in   C^{\alpha} (\Sigma)$ and $y = G(u)$ 	we have
	\begin{equation} \label{fdash1}
	\langle F_i'(u), v\rangle =-\int_{\Sigma}  \frac{\partial \varphi }{\partial u}
	(\cdot, y, u)w_i  v \; ds dt 
	\end{equation} 
	where $w_i$ is the solution of
	\begin{eqnarray} \label{fdash2} -D_t w-\Delta w  = 	\pd{a_i}{y}(\cdot, y)    \quad   \mbox{in} \; Q,  \qquad    w(x, T)  =  0  \quad \mbox{in} \;  \Omega.  \nonumber \\
	-D_t w-  \Delta_{\Gamma} w +  \partial_{\nu} w =  \frac{\partial \varphi }{\partial y}
	(\cdot, y, u)  w - \pd{b_i}{y}(\cdot,y, u)  \quad \mbox{on} \;   \Sigma .
	\end{eqnarray}
\end{theorem}

\begin{proof} The proof is similar to that of Theorem \ref{diffj}.  Starting with the mapping $K_i: C^{2,\alpha} (Q) \times  C^{\alpha} (\Sigma) \rightarrow \RR$ given by
	\[ K_i(y) =\int_{Q} a_i(\cdot, y) \; dx dt +\int_{\Sigma} b_i(\cdot, y) \; ds dt
	\]
we have for every $\tilde{y}$ and $y$ 
	\[ \langle \pd{K_i}{y}(y), \tilde{y}\rangle = \int_{Q} \pd{a_i}{y}(\cdot, y) \tilde{y} \; dx dt
	+\int_{\Sigma} \pd{b_i}{y}(\cdot, y) \tilde{y}\; ds dt.
	\]
Since $ F_i(u) = K_i(G(u))$, by the chain rule we have
	\[ \langle F_i'(u), v \rangle = \langle \pd{K_i}{y}(y)G'(u), v\rangle 
	= \langle \pd{K_i}{y}(y), G'(u) v \rangle \]
where $G'(u) v $ stands for the solution $z= \langle G'(u),  v \rangle$ of (\ref{gdash1}) in Theorem \ref{diffg}. Therefore
	\begin{eqnarray} 
	\langle F_i'(u), v\rangle	& = &  \int_{Q} \pd{a_i}{y}(\cdot, y)z \; dx dt
	+ \int_{\Sigma} \pd{b_i}{y}(\cdot, y, u) z \; ds dt.
	\end{eqnarray}
	
Now we assign (\ref{gdash1}) to be the principal system and treat (\ref{fdash1}) as its adjoint system.  Let $w_i$ be the solution of (\ref{fdash1}).  Applying Theorem \ref{adjointeqn} to $z$ and $w_i$ together with the information $z_0=0$,  $w_T=0 $,
	\[ f=0, \quad g=\pd{a_i}{y}(\cdot, y), \quad h = \frac{\partial \varphi }{\partial y} (\cdot, y, u)  z + \frac{\partial \varphi }{\partial u}
	(\cdot, y, u) v, \qand r = \frac{\partial \varphi }{\partial y}
	(\cdot, y, u)  w_i - \pd{b_i}{y}(\cdot,y) \]
	we obtain
	\begin{eqnarray*} 
		\int_{\Sigma} z \left( \frac{\partial \varphi }{\partial y}
		(\cdot, y, u)  w_i - \pd{b_i}{y}(\cdot,y)\right)      \; ds dt    &=&  \int_{Q} z \pd{a_i}{y}(\cdot, y)  \; dx dt  \\
		& & + \int_{\Sigma} w_i  \left( \frac{\partial \varphi }{\partial y} (\cdot, y, u)  z + \frac{\partial \varphi }{\partial u}
		(\cdot, y, u) v\right)    \; ds dt   
	\end{eqnarray*}
which is 
	\[ \int_{Q} z \pd{a_i}{y}(\cdot, y)  \; dx dt = -\int_{\Sigma} z \pd{b_i}{y}(\cdot,y, u))     \; ds dt   - \int_{\Sigma} w_i \frac{\partial \varphi }{\partial u}
	(\cdot, y, u) v    \; ds dt  \] 
	A substitution of this into (\ref{fdash1}) gives
	\[ \langle F_i'(u), v\rangle  = -\int_{\Sigma}  \frac{\partial \varphi }{\partial u}
	(\cdot, y, u)w_i v \; ds dt.
	\]
\end{proof}

We are now in the position to state the first order necessary condition by using the KKT-condition.  This requires some regularity conditions on the derivatives of $F_i(u)$.  Different versions of regularity conditions are developed as frame works for general problems in a Banach space, such as \cite{robinson}  and \cite{zowe}.  We simply adopt the regularity conditions used by Casas and  Tr\"{o}ltzsch in \cite{casas1}.  Assume $\bar{u}$ is a local solution and $\bar{y}$ is the corresponding state.  Let $ I_0$ be the active index set such that $F_i(\bar{y}) =0$ if $i \in I_0$ and $F_i(\bar{y}) <0$ if $i$ is not in $I_0$.  The regularity assumption is as follows.
	\begin{equation} \label{regular}
	 \exists \{ h_i \}_{i \in I_0} \subset  C^{\alpha} (\Sigma) \quad \mbox{with} \quad \mbox{supp} \; h_i \subset \Gamma_{\epsilon} \quad \mbox{such that} \quad
 \langle F'_i(\bar{u}) , h_j \rangle =\delta_{ij},
		\end{equation}
where $\Gamma_{\epsilon}$ is the subset of $\Gamma$ on which the function values of $\bar{u}$ are away from the boundary of the allowable control set $U$ by a amount $\epsilon >0$: 
 \[ \Gamma_{\epsilon} = \{ (x, t) \in \Gamma \; | \; u_a (x, t) +\epsilon \leq \bar{u} (x, t) \leq u_b (x, t) -\epsilon \}. \] 
\begin{theorem} Assume that the regularity condition (\ref{regular}) holds.  Then the first necessary condition for $\bar{u}$ to be a local optimal solution of $\inf J(u)$ is that for $i \in I_0$ there exist real numbers $\lambda_i \geq 0$  and solutions $\bar{y}$ and $\bar{w} $ of the following system
	\begin{equation} \label{neconsys}
	\left\{ \begin{array}{l}
		\dst	D_t y-\Delta y = f  \quad  \mbox{in}  \;  Q, \qquad 
	D_t y-  \Delta_{\Gamma} y +  \partial_{\nu} y =\varphi (\cdot , y, \bar{u})   \quad \mbox{on}  \;   \Sigma , \qquad 
	y(x, 0) = y_0  \quad  \mbox{in}  \;  \Omega. \\ 
	\phantom{x} \\	
	\dst		-D_t w-\Delta w  = 	\pd{p}{y}(\cdot, y) + \sum_{i\in I_0} \lambda_i	\pd{a_i}{y}(\cdot, y)    \quad   \mbox{in} \; Q,  	\quad \quad   w(x, T)  =  0  \quad \mbox{in} \;  \Omega \\ 
\dst -D_t w-  \Delta_{\Gamma} w +  \partial_{\nu} w =   \frac{\partial \varphi }{\partial y}
	(\cdot, y, \bar{u})  w - \pd{q}{y}(\cdot,y, \bar{u}) - \sum_{i\in I_0} \lambda_i \pd{b_i}{y}(\cdot,y, \bar{u}) \quad \mbox{on} \;   \Sigma. 
	\end{array}	\right.
	\end{equation}	
such that 
	\begin{equation} \label{necond}
	 \int_{\Sigma} \left( \pd{q}{u}(\cdot,\bar{y} ,\bar{u})- \frac{\partial \varphi }{\partial u}
	(\cdot, \bar{y},\bar{u})\bar{w} \right) v \; ds dt \geq 0
	\end{equation}	
holds for all $v$ satisfying $\| v \|_{ C^{\alpha} (\Sigma)} \leq \delta$ and $ \bar{u} +v \in U$.
\end{theorem}

\begin{proof}  Let $y = G(u)$ be the solution of (\ref{state}) corresponding to $u$ and $F_i (u) = F_(G(u))$ are the constraint functionals in our optimal control problem (\ref{prob}). By the KKT conditions, for $i \in I_0$, there are $\lambda_i \geq 0$ such that
	\begin{equation} \label{lag}
	\langle J'(\bar{u})+ \sum_{i\in I_0} \lambda_i F_i '(\bar{u}), v \rangle
	\geq 0
	\end{equation}
for all  $v$ as described in the Theorem.  

Now all we need to show is that the right hand side of (\ref{necond}) is the expression for $\langle J'(\bar{u})+ \sum_{i\in I_0} \lambda_i F_i '(\bar{u}), v \rangle$.  Putting $u=\bar{u}$ $\bar{y} = G(\bar{u})$ in equations (\ref{jdash2}) and  (\ref{fdash2}) we obtain solutions $w_0$ and $w_i$ respectively.  It is obvious that $\bar{w}= w_0+  \sum_{i\in I_0} \lambda_i w_i$ is the solution of the second equation in system (\ref{neconsys}).  Also (\ref{jdash1}) gives
\[ 	\langle J'(\bar{u}), v \rangle =  \int_{\Sigma} \left( \pd{q}{u}(\cdot,\bar{y} ,\bar{u})- \frac{\partial \varphi }{\partial u}
(\cdot, \bar{y},\bar{u}) w_0 \right) v \; ds dt \]
and  (\ref{fdash1})  gives
 \[ \langle F_i'(\bar{u}), v \rangle =  - \int_{\Sigma}  \frac{\partial \varphi }{\partial u}
 (\cdot, \bar{y},\bar{u}) w_i v \; ds dt . \]
Adding up the above two equations we obtain
	\[  \langle J'(\bar{u})+ \sum_{i\in I_0} \lambda_i  F_i'(\bar{u}),v \rangle
	=  \int_{\Sigma} \left( \pd{q}{u}(\cdot,\bar{y} ,\bar{u})- \frac{\partial \varphi }{\partial u}
	(\cdot, \bar{y},\bar{u})\bar{w} \right) v \; ds dt.\]
\end{proof}

Finally we point out that a second order necessary optimality condition for the problem (\ref{prob}) can be directly translated from the elliptic case \cite{casas1} to our parabolic case. To keep this article short we will not state it here and the interested readers can find it in \cite{casas1}.

\section{Remarks}
\label{sect:5}

There is no problem to extend the results in this paper to the case when the state equation is a general second order parabolic equation with a general Venttsel boundary condition:
\begin{eqnarray} \label{stateg}
D_t y-  a^{ij} D_{ij}y + b^i D_i y + c y = f
\qquad &
\mbox{in}  & \quad  Q, \nonumber \\
D_t y -  \alpha^{ij}\partial_{i} \partial_{j} y+  \beta ^i \partial_{\nu} y  = \varphi (\cdot, y,u)  \quad &
\mbox{on}  & \quad   \Gamma,  \nonumber \\
y(x, 0 )= y_0  \quad &
\mbox{in}  & \quad   \Omega
\end{eqnarray}
where   $\{a^{ij}\}$ and $\{\alpha^{ij}\}$ are both positive definite matrices and $\beta ^i>0$. In this general case, similar results to Theorem \ref{diffj} and \ref{difff} can be established but the formulations will be messy.  The reason for this is that we have to perform local coordinate transform to change  $\{a^{ij}\}$ and $\{\alpha^{ij}\}$ into identity matrices before we can apply the Green's formula both in $\Omega$ and on $\Gamma$. 

In engineering problems the control function $u$ is usually discontinuous. In particular, when changing control strategy $u$ usually has a jump at certain time. We are currently working on the Vettsel problem in Hilbert spaces $H^s(Q)$ so that optimality conditions for discontinuous controls will be covered.


\begin{thebibliography}{}

%
\bibitem{apush:1} Apushkinskaya D. E. and Nazarov A. I., {\em A survey of results on nonlinear Venttsel problems}, Applications of Mathematics, 45, No. 1, 69-80, (2000)

\bibitem{apush:2} Apushkinskaya, D. E.; Nazarov, A. I. {\em Hölder estimates of solutions to initial-boundary value problems for parabolic equations of nondivergent form with Wentzel boundary condition},  Amer. Math. Soc. Transl. (2) 64, 1-13, (1995)

\bibitem{apush:3} Apushkinskaya, D. E.; Nazarov, A. I. {The nonstationary Venttsel problem with quadratic growth with respect to the gradient},  J. Math. Sciences,  Vol. 80,  No. 6,  2197–2207, (1996)

\bibitem{ben} Ben Tal A.  and  Zowe J., {\em A unified theory of first and second order conditions
	for extremum problems in topological vector spaces}, Math. Programming Study, 19, 39-76, (1999)

\bibitem{bon} Bonnans J.  and  Casas E., {\em Contr\^{o}le de syst\`{e}mes elliptiques 	semilineaires comportant des contraintes sur l'\`{e}tat}, In Nonlinear Partial Differential Equations and Their Applications, Coll\`{e}ge de France Seminar, H. Brezis and J. lions, eds., vol
8, 69-86, Londonman Scientific \& Technical, New York, (1988)

\bibitem{cj} Carslaw, H. S. and Jaeger, J. C., {\em Conduction of heat in solids }, Oxford, Clarendon Press, (1959)

\bibitem{casas1} Casas E.  and  Tr\"{o}ltzsch F., {\em Second-order necessary optimality conditions for some state-Constrained control problems of 	semilinear elliptic equations}, Appl Math Optim, 39, 211-227, (1999)

\bibitem{casas2} Casas E.  and  Tr\"{o}ltzsch F., {\em Second-order sufficient optimality conditions for some state-Constrained control problems of semilinear elliptic equations}, SIAM J. Control Optim., 
Vol. 38, No. 5, pp. 1369–1391 (2000)

\bibitem{friedman} Friedman A., {\em Partial differential equations of parabolic type}, Dover Publications, New York, (1992)

\bibitem{GilbTrud:1} D. Gilbarg and N. S. Trudinger, \emph{Elliptic Partial Differential Equations of the Second Order}, 2nd Edition, Springer Verlag, (1975)

\bibitem{Krum} Krumbiegel K. and Rehberg J. \emph{Second order sufficient optimality conditions for parabolic optimal control problems with pointwise state constraints}, SIAM J. Control Optim., Vol. 51, No. 1, 304–331, (2013)

\bibitem{lions}	Lions J.-L., {\em Optimal Control of Systems Governed by Partial Differential Equations}, Springer, 1971

\bibitem{luo:1}	Luo Y., {\em Necessary optimality conditions for some control problems of elliptic equations with Venttsel boundary conditions}, Appl. Math. Optim., 61, 337-351, (2001)

\bibitem{luo:2} Luo Y. and Trudinger N.S., {\em Linear second order elliptic equations with Venttsel boundary conditions}, Proc. Royal Society of Edinburgh, 118A, 193-207, (1991)

\bibitem{luo:3} Luo Y. {\em Quasilinear second order elliptic equations with elliptic Venttsel boundary conditions}, Nonlinear Anal. 16, 761-769, (1991)

\bibitem{luo:4}	Luo Y., {\em The heat conduction in a medium enclosed by a thin shell of higher diffusivity}, Proc. 5th Colloquium on Diff. Equations, Bulgaria, (1994)

\bibitem{Raymond} Raymond J.-P. and Tr\"{o}ltzsch F.-J. {\em Second order sufficient optimality conditions for nonlinear parabolic optimal control problems with state constraints}, Discrete and Continuous Dynamical Systems, Vol. 6, No. 2, 431–450, (2000)

\bibitem{ven} Venttsel A. D., \emph{On boundary conditions for multidimensional diffusion processes}, Theor. Probab. Appl., 4, 164-177, (1959)

\bibitem{robinson} Robinson S. M., \emph{First order conditions for general nonlinear optimization}, SIAM J. Appl. Math., Vol. 30 No.4, 597-607, (1976)



\bibitem{zowe} Zowe L. and Kurcyusz S., \emph{Regularity and stability for the mathematical programming problem in Babach spaces},  Appl. Math. Optim., 5, 49-62, (1979)

\end{thebibliography}
\end{document}